\documentclass[a4paper,12pt]{amsart}
\usepackage[a4paper, total={6.5in, 10.5in}]{geometry}
\usepackage{amsmath, amsthm, amssymb, amsopn, amscd, mathtools, array, enumerate, amsfonts, graphics, hyperref, wasysym, chngcntr, linguex, graphicx,ytableau,tikz,mathrsfs,verbatim,algorithm,algorithmic,multicol}

\newtheorem{theorem}{Theorem}[section]
\newtheorem{corollary}[theorem]{Corollary}
\newtheorem{lemma}[theorem]{Lemma}

\newtheorem{lemma and definition}[theorem]{Lemma and Definition}
\newtheorem{proposition}[theorem]{Proposition}
\newtheorem{definition}[theorem]{Definition}
\newtheorem{exam}[theorem]{Example}

\newtheorem{remark}[theorem]{Remark}

\newtheorem{the construction}[theorem]{THE CONSTRUCTION}

\newcommand{\field}[1]{\mathbb{#1}}

\newcommand{\Z }{\field{Z}}
\newcommand{\N }{\field{N}}

\DeclareMathOperator{\G}{G}
\DeclareMathOperator{\g}{g}
\DeclareMathOperator{\F}{F}

\DeclareMathOperator{\PF}{PF}

\DeclareMathOperator{\pf}{PF}

\DeclareMathOperator{\co}{C}

\DeclareMathOperator{\ty}{t}

\DeclareMathOperator{\ft}{Filtered}
\DeclareMathOperator{\len}{Length}

\begin{document}
	
	\title{Young Diagram Decompositions for Almost Symmetric Numerical Semigroups}
	%\author{Manuel Delgado}
	%\address{Department of Mathematics, Faculty of Science, University of Porto, Porto, Portugal}
	%\email{mdelgado@fc.up.pt}
	\author{Mehmet YEŞİL}
	\address{Department of Mathematics, Faculty of Science and Letters, Batman University, Batman, Turkey}
	\email[corresponding author]{mehmet-yesil@outlook.com}
	\maketitle
	
	\begin{abstract}
		This paper introduces new structural decompositions for almost symmetric numerical semigroups through the combinatorial lens of Young diagrams. To do that, we use the foundational correspondence between numerical sets and Young diagrams, which enables a visual and algorithmic approach to studying properties of numerical semigroups. Central to the paper, a decomposition theorem for almost symmetric numerical semigroups is proved, which reveals that such semigroups can be uniquely expressed as a combination of a numerical semigroup, its dual and an ordinary numerical semigroup.
	\end{abstract}

\section{Introduction}

Throughout, we denote by $\Z$ and $\N$ the sets of integers and positive integers, respectively, and define $\N_{0}=\N\cup\{0\}$. A \textit{numerical set} $R$ is a subset of $\N_{0}$ that contains $0$ and has a finite complement in $\N_{0}$. In particular, $\N_{0}$ itself is a numerical set with an empty complement. If a numerical set $R$ is closed under addition, i.e. if $x,y \in R$ implies $x+y\in R$, then it is referred to as a \textit{numerical semigroup}.

Numerical semigroups naturally emerge in various areas of mathematics, such as algebraic geometry, commutative algebra, and combinatorics. They provide a rich framework for studying problems related to factorization, symmetry, and combinatorial structures. A central theme in the study of numerical semigroups is the classification and decomposition of semigroups with specific properties which are often closely tied to the structure of their gaps, i.e. the elements of $\N_0$ not contained in the semigroup.

A powerful tool for visualizing and analysing numerical sets and semigroups is the use of Young diagrams. A Young diagram is a combinatorial object consisting of left-aligned rows of boxes, with the number of boxes in each row non-decreasing from bottom to top. A bijective correspondence between Young diagrams and proper numerical sets was introduced by Keith and Nath in \cite{KN}. This correspondence provides a geometric perspective on semigroup properties, facilitating intuitive and algorithmic approaches to their study. Recent research has further explored this connection, uncovering deep links between the algebraic structure of semigroups and the combinatorics of their associated diagrams.

Constantin, Houston-Edwards and Kaplan utilized this correspondence to derive combinatorial results about core partitions in \cite{HBN}. Tutaş and her collaborators characterized numerous algebraic properties of Arf numerical semigroups using this framework in a series of works \cite{N4,N3,N2,N5,N1}. In particular, Gümüşbaş and Tutaş introduced primitive semigroup decompositions in \cite{N4}, while Karakaş and Tutaş developed hook set decompositions for Arf semigroups in \cite{N2}.

Furthermore, Süer and Yeşil used the correspondence to define Young diagram decompositions for classically irreducible numerical semigroups into a semigroup and its dual in \cite{SY}, and to introduce special subdiagrams of Young diagrams in \cite{SY2}.

In this paper, we focus on almost symmetric numerical semigroups, a class that generalizes the concept of symmetric semigroups. Almost symmetric semigroups were introduced by Barucci and Fröberg in \cite{BF} in the context of generalizing one-dimensional Gorenstein rings. These semigroups occupy a central position in the theory of numerical semigroups due to their algebraic richness and structural tractability, and they have been the subject of significant research in recent years.

The main contribution of this paper is a decomposition theorem for almost symmetric numerical semigroups, showing that such semigroups can be uniquely expressed as combinations of a numerical semigroup, its dual, and an ordinary numerical semigroup. This result extends the work of Süer and Yeşil in \cite{SY}, generalizing their decomposition approach to a broader class: almost symmetric numerical semigroups whose pseudo-Frobenius numbers are consecutive, except for the Frobenius number itself.

This paper is organized as follows. In Section \ref{sec2}, we provide the preliminary concepts of numerical sets and semigroups, including gaps, the Frobenius number, pseudo-Frobenius numbers, duality and other related notions. Section \ref{sec3} establishes the correspondence between numerical sets and Young diagrams, introducing key algorithms for translating between these two representations. In Section \ref{sec4}, we define bonded, end-to-end, and conjoint sums of Young diagrams and numerical sets, which form the basis for our decomposition techniques. We also provide the semigroup structures of the set of proper numerical sets under these operations, and duality relations on them. Section \ref{sec5} presents main results of the paper, including the symmetric numerical set construction lemma (see Lemma \ref{sym}), and the decomposition theorem (see Theorem \ref{mainthm}) for almost symmetric numerical semigroups, accompanied by illustrative examples.

\section{Numerical Sets and Semigroups}\label{sec2}

If a numerical set $R$ is not equal to $\N_{0}$, it is referred to as \textit{proper}. For a proper numerical set $R$, the complement $\G(R)=\N_{0}\setminus R$ is called the set of \textit{gaps} of $R$. The \textit{genus} of $R$, denoted by $\g(R)$, is the number of gaps of $R$, i.e. the cardinality of $\G(R)$. The largest element of $\G(R)$ is known as the \textit{Frobenius number} of $R$, denoted by $\F(R)$. The quantity $\co(R)=\F(R)+1$ is referred to as the \textit{conductor} of $R$, which satisfies that $z\in\N_{0}$ and $z>\co(R)$ implies $z\in R$. Also, we define the Frobenius of $\N_{0}$ to be $-1$, and its conductor to be $0$.

The elements of a proper numerical set $R$ that are less than or equal to the conductor $\co(R)$ are referred to as the \textit{small elements} of $R$. If $R$ has $n$ non-zero small elements, we list them as $0=r_{0}<r_{1}<\dots<r_{n-1}<r_n=\co(R)$, and write $$R=\{0,r_{1},\dots,r_{n-1},r_{n}=\co(R),\rightarrow\},$$ where meaning of the arrow at the end is that all integers greater than $\co(R)$ are also elements of $R$. The small elements of $R$ that are less than the conductor are called the \textit{left elements} of $R$.

A numerical set $R$ is called \textit{ordinary} if it has exactly one non-zero small element. More generally, if $R$ has $n$ non-zero small elements and genus $\g(R)=g$, then it follows that $\co(R)=n+g$ and $\F(R)=n+g-1$.

\begin{exam}\label{ex1}
	The set $R=\{ 0,4,6,7,9,10,12\rightarrow \}$ is a numerical set with $\F(R)=11$ and $\co(R)=12$, where its gap set is $\G(R)=\{ 1,2,3,5,8,11\}$ and genus $\g(R)=6$.
\end{exam}

A numerical set $R$ is said to be \textit{symmetric} if for all $a \in [0,\F(R)]\cap\N_0$, exactly one of the elements $a$ or $\F(R)-a$ belongs to $R$.

Let $R=\{0,r_{1},\dots,r_{n-1},r_{n}=\co(R),\rightarrow\}$ be a numerical set with $n$ non-zero small elements, and let $\G(R)=\{ a_{1},\dots,a_{g}=\F(R) \}$ be its set of gaps. The \textit{dual} of $R$, denoted by $R^{*}$, is defined as  $$R^{*}=\{0,\F(R)-a_{g-1},\F(R)-a_{g-2},\dots,\F(R)-a_{1},r_{n},\rightarrow \}$$ where $\co(R^{*})=\co(R)$, and hence $\F(R^{*})=\F(R)$. By construction, $R^{*}$ is a numerical set with $g$ non-zero small elements and $$\G(R^{*})=\{ \F(R)-r_{n-1},\F(R)-r_{n-2},\dots,\F(R)-r_1,\F(R) \}.$$ Notice also that $R=(R^{*})^{*}$, and that $R=R^{*}$ if and only if $R$ is symmetric.

Let $R$ be a numerical semigroup and let $I\subset R$. If $$R=\{ \sum_{i=1}^{e}\lambda_{i}n_{i} \mid n_{i}\in I,\lambda_{i}\in\N_{0},e\in \N \},$$ then it is said that $R$ is \textit{generated} by $I$, and written as $R=\langle I \rangle$. Also, it is commonly known that every numerical semigroup is finitely generated; that is there always exists a finite set $I$ such that $R=\langle I \rangle$. If there is no proper subset of $I$ that generates $R$, then $I$ is called a \textit{minimal generating set} of $R$. The number of element in a minimal generating set of $R$ is referred to as \textit{embedding dimension} of $R$.

An element $z\in\Z$ is said to be a \textit{pseudo-Frobenius number} of $R$ if $z\notin R$ and $z+r\in R$ for all non-zero element $r\in R$. The set of pseudo-Frobenius numbers of $R$ is denoted by $\pf(R)$. The cardinality of $\PF(R)$ is referred to as the \textit{type} of $R$, denoted by $\ty(R)$.

%For $z,w\in\Z$, we say that $z\leq_{S}w$ if $w-z\in S$, which defines a partially ordered relation. For a subset $A$ of $S$, $\maxx_{\leq_S}(A)$ will denote the maximal elements of $A$ with respect to the relation $\leq_{S}$. By the definition of pseudo-Frobenius numbers, we can see that they are the maximal elements in $\Z\setminus S$ with respect to $\leq_S$, i.e. $\pf(S)=\maxx_{\leq_S}(\Z\setminus S)$. Then we can easily deduct that $z\in\Z\setminus S$ if and only if $f-z\in S$ for some $f\in\pf(S)$.

A numerical semigroup $R$ is symmetric if it is symmetric as a numerical set, or equivalently if it has the property that  $2\g(R)=\F(R)+1$. And $R$ is called \textit{pseudo-symmetric} if and only if $\F(R)$ is even and for every $z\in \Z\setminus R$, either $z\in \{ \frac{\F(R)}{2}, \F(R) \}$ or $\F(R)-z\in R$. Equivalently $R$ is pseudo-symmetric if and only if $2\g(R)=\F(R)+2$.

More generally, it is well-known that $2\g(R)\geq \F(R)+\ty(R)$ is valid for any numerical semigroup $R$. When the equality holds, i.e. $2\g(R)= \F(R)+\ty(R)$, $R$ is said to be \textit{almost symmetric}. Moreover, $R$ is almost symmetric if and only if  every $z\in \Z\setminus R$ satisfies that either $z\in\pf(R)$ or $\F(R)-z\in R$.

For more intuition on numerical semigroups, we refer to \cite{RG}.

\section{Numerical Sets and Young Diagrams}\label{sec3}

Let $Y$ be a Young diagram with $n$ columns and $g$ rows. The number of boxes in a given column (respectively, row) is referred to as the \textit{length} of that column (respectively, row). The \textit{hook} of a box in $Y$ consists of the box itself, all boxes to its right in the same row, and all boxes below it in the same column. The total number of boxes in the hook is called the \textit{hook length} of that box. For more intuition on Young diagrams, we refer to \cite{F}.

Let $R$ be a numerical set. The associated Young diagram $Y_{R}$ can be constructed by tracing a continuous polygonal path in $\Z^{2}$ starting from the origin. Beginning with $r=0$, we proceed as follows:
\begin{enumerate}
	\item if $r \in R$, draw a unit-length line segment to the right,
	\item if $r \notin R$, draw a unit-length line upward.
\end{enumerate}
This process is repeated incrementally for $r+1$ until reaching $r=\F(R)$. The region bounded above by the horizontal line at height $\g(R)$, below by the constructed polygonal path and left by the $y$-axis defines the Young diagram $Y_{R}$. It is evident that every Young diagram has such a polygonal path and it uniquely corresponds to a proper numerical set. Thus, the correspondence $S\rightarrow Y_{R}$ establishes a bijection between the set of Young diagrams and the set of proper numerical sets.

\begin{exam} \label{ex2}
	The Young diagram depicted below is associated with the numerical set given in Example \ref{ex1}.
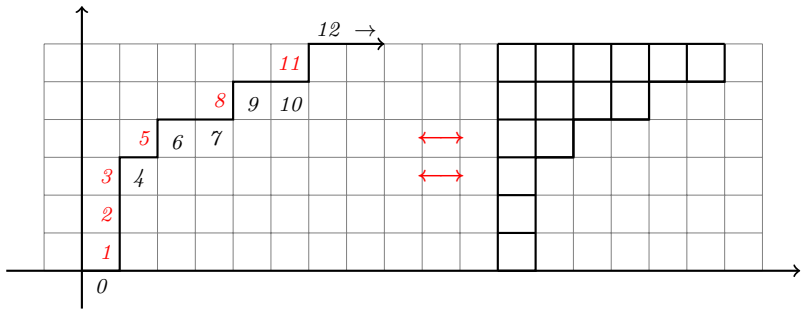
	\begin{center}
		\tiny{\begin{tikzpicture}[scale=0.5]
				\draw[very thin, gray] (-1,0) grid (18,6);
				\draw[thick, ->] (-2,0) -- (19,0) node[anchor=north west] {};
				\draw[thick, ->] (0,-1) -- (0,7) node[anchor=south east] {};
				\draw[thick, ->] (0,0)--(1,0)--(1,3)--(2,3)--(2,4)--(4,4)--(4,5)--(6,5)--(6,6)--(8,6);
				
				\node at (0.5,0)[below]{0};
				\node at (1,0.5)[left]{\color{red}{1}};
				\node at (1,1.5)[left]{\color{red}2};
				\node at (1,2.5)[left]{\color{red}3};
				\node at (1.5,2)[above]{4};
				\node at (2,3.5)[left]{\color{red}5};
				\node at (2.5,3)[above]{6};
				\node at (3.5,3)[above]{7};
				\node at (4,4.5)[left]{\color{red}8};
				\node at (4.5,4)[above]{9};
				\node at (5.5,4)[above]{10};
				\node at (6,5.5)[left]{\color{red}11};
				\node at (6.5,6)[above]{12};
				\node at (7.5,6)[above]{$\rightarrow$};
				
				\node at (9.5,2.5){$\color{red}\boldsymbol{\longleftrightarrow}$};
				\node at (9.5,3.5){$\color{red}\boldsymbol{\longleftrightarrow}$};
				\draw[thick, -] (11,0) -- (11,6);
				\draw[thick, -] (11,0)--(12,0)--(12,3)--(13,3)--(13,4)--(15,4)--(15,5)--(17,5)--(17,6);
				\draw[thick, -] (11,6)--(17,6);
				\draw[thick, -] (11,5)--(17,5);
				\draw[thick, -] (11,4)--(15,4);
				\draw[thick, -] (11,3)--(13,3);
				\draw[thick, -] (11,2)--(12,2);
				\draw[thick, -] (11,1)--(12,1);
				\draw[thick, -] (12,3)--(12,6);
				\draw[thick, -] (13,3)--(13,6);
				\draw[thick, -] (14,4)--(14,6);
				\draw[thick, -] (15,4)--(15,6);
				\draw[thick, -] (16,5)--(16,6);
		\end{tikzpicture}}
	\end{center}
\end{exam}

Let $R=\{0,r_{1},\dots,r_{n-1},r_{n},\rightarrow\}$ be a numerical set with gap set $\G(R)=\{ a_{1},\dots,a_{g} \}$ and associated Young diagram $Y_{R}$. By construction, $Y_{R}$ has $g$ rows and $n$ columns. From left to right, each $i$th small element of $R$ corresponds to the $(i+1)$th column of $Y_R$, so that $r_0$ corresponds to the first column, $r_1$ to the second column, and so on. Similarly, from bottom to top, for each gap $a_j$ of $R$ corresponds to the $j$th row of $Y_R$.

The hook length of the box in the $i$th row (from bottom to top) and the first column of $Y_R$ is the $i$th gap $a_i$ of $R$. Moreover, for each $i\in [0,n-1]\cap\N_0$, the $i$th column of $Y_{R}$ can be identified with the set of hook lengths of the boxes in it. We denote this set by $\G_{i}(R)$.

\begin{proposition} \label{prp} \cite[Section 2]{N1}
	Let $R=\{0,r_{1},\dots,r_{n-1},r_{n},\rightarrow\}$ be a numerical set with the associated Young diagram $Y_{R}$. Then:
	\begin{enumerate}
		\item For each $i\in [0,n-1]\cap\N_0$, we have $$\G_{i}(R)=\{ a-r_i \mid a\in \G(R) \text{ and } a>r_i \}=\G(R-r_i)$$ where $R-r_i=\{ 0,r_{i+1}-r_i,\dots,r_n-r_i,\rightarrow \}$.
		\item $R$ is a numerical semigroup if and only if $\G_{i}(R)\subseteq \G_{0}(R)$ for all $i\in [0,n-1]\cap\N_0$, or equivalently $\bigcup_{i=0}^{n-1}\G_{i}(R)=\G(R)$.
	\end{enumerate}
\end{proposition}

By Proposition \ref{prp}, it follows that $\G_{0}(R)=\G(R)$, and for each $i\in [0,n-1]\cap\N_0$, the hook length of the box in the $i$th column (from left to right) and the top row of $Y_{R}$ is $\F(R)-r_{i}$.

\begin{exam}
	The associated Young diagram $Y_{R}$ of the numerical set $R$ given in Example \ref{ex1} has $4$ columns and $4$ rows, and the hook lengths of the boxes in $Y_{R}$ are shown in the picture below.
	\vspace{1em}
	\begin{center}
		\tiny{\ytableaushort
			{{11}75421,8421,51,3,2,1}
			*{6,4,2,1,1,1}}
	\end{center}
	\vspace{1em}
\end{exam}

For a Young diagram $Y$, we define the \textit{dual} of $Y$ as the Young diagram which is obtained from interchanging the rows and columns of $Y$, and denote it as $Y^{*}$. That is, the $i$th row of $Y$ becomes the $i$th column of $Y^{*}$. Notice that $Y=(Y^{*})^{*}$. If $Y_{R}$ is the Young diagram associated to $R$, then $Y_{R}^{*}=Y_{R^{*}}$ is the Young diagram associated to $R^{*}$.

A Young diagram $Y$ is said to be \textit{symmetric} if the associated numerical set is symmetric. Equivalently, $Y$ is symmetric if and only if $Y=Y^{*}$.

For a positive integer $N$, a \textit{partition} $\nu$ of $N$ is a finite non-increasing sequence of positive integers $\nu_1\geq\nu_2\geq\dots\geq\nu_n$ such that $\nu_1+\nu_2+\dots+\nu_n=N$, denoted by $\nu=[\nu_1,\nu_2,\dots,\nu_n]$. Each $\nu_i$ is called a \textit{part} of the partition and the number $n$ of parts is referred to as the \textit{length} of the partition.

In the context of Young diagrams, listing all the lengths of all rows from top to down yields a partition. Conversely, every partition $\nu=[\nu_1,\nu_2,\dots,\nu_n]$ corresponds to a Young diagram with $\nu_1$ columns and $n$ rows, where the length of $i$th row (from top to down) is $\nu_i$. This correspondence defines a bijective mapping between the set of partitions and the set of Young diagrams. Hence, we can identify a Young diagram $Y$ as $Y=\nu$ if its corresponding partition is $\nu$. For instance, the associated Young diagram $Y_R$ of the numerical set $R$ given in \ref{ex1} is corresponding to the partition $[6,4,2,1,1,1]$ of $15$, and so we write $Y_R=[6,4,2,1,1,1]$.

Next, we present algorithms for visualizing numerical sets using Young diagrams. To facilitate their description, we begin by establishing some notation. Let $L=[l_1,\dots,l_n]$ be a list of length $n$. We denote by $L[i]$ the $i$th element of $L$.  For an integer $a$, we define $a\pm L$ as the list $[a\pm l_1,\dots,a\pm l_n]$. Additionally, $\ft(L,m\rightarrow m>0)$ denotes the list containing only the positive elements of $L$, preserving their original order.

\begin{algorithm} 
	\caption{Algorithm for finding the associated partition of a numerical set.}
	\label{alg1}
	\begin{algorithmic}
		\REQUIRE A list of integers starting from zero in increasing order $[r_0=0,r_1,\dots,r_n]$.
		\ENSURE The associated partition of the numerical set $R=\{ 0,r_1,\dots,r_n,\rightarrow \}$.
		\STATE \textbf{Initialize} $P=[\dots]$ and $i=n$
		\WHILE{$i>0$}
		\STATE $L=[\dots]$
		\STATE $q_i=r_i-r_{i-1}-1$
		\FOR{$k\in \{1,\dots,q_i\}$}
		\STATE $L[k]=i$
		\ENDFOR
		\STATE Append $L$ to $P$
		\STATE $i\leftarrow i-1$
		\ENDWHILE
		\RETURN $P$
	\end{algorithmic}
\end{algorithm}

Algorithm \ref{alg1} takes the list of small elements of a numerical set $R$ and returns a partition, where the parts represent the row lengths of the Young diagrams $Y_R$ associated with $R$. For example, if $R=\{ 0,4,6,7,9,10,12,\rightarrow \}$, the output of Algorithm \ref{alg1} for $R$ is the partition $[6,4,2,1,1,1]$. To draw the associated Young diagram $Y_R$ in LaTeX, one can use \texttt{ytableau} package and place the output inside the command \verb|\ydiagram{...}|.

\begin{algorithm} 
	\caption{Algorithm for finding the associated numerical set of a partition.}
	\label{alg2}
	\begin{algorithmic}
		\REQUIRE A list of non-increasing positive integers $[p_1,\dots,p_n]$.
		\ENSURE The associated numerical set of the partition $P=[p_1,\dots,p_n]$.
		\STATE \textbf{Initialize} $R=[0]$, $i=P[n]=p_n$ and $j=1$
		\IF{$i>1$}
		\STATE Append $[1,\dots,i-1]$ to $R$
		\ENDIF
		\WHILE{$j<n$}
		\STATE $a=P[n-j]-P[n-j+1]$
		\IF{$a=0$}
		\STATE $i\leftarrow i+1$
		\ELSE
		\STATE Append $[i+1,\dots,i+a]$ to $R$
		\STATE $i\leftarrow i+a+1$
		\ENDIF
		\STATE $j\leftarrow j+1$
		\ENDWHILE
		\STATE Add $i+1$ to $R$
		\RETURN $R$
	\end{algorithmic}
\end{algorithm}

Algorithm \ref{alg2} performs the inverse of Algorithm \ref{alg1}. Given a partition $P$, it returns the list of small elements of the associated numerical set. For instance, if $P=[6,4,2,1,1,1]$, then the output of Algorithm \ref{alg2} for $P$ is $[0,4,6,7,9,10,12]$.

\begin{algorithm} 
	\caption{Algorithm for finding the hook lengths of the associated Young diagram of a numerical set.}
	\label{alg3}
	\begin{algorithmic}
		\REQUIRE A list of integers starting from zero in increasing order $[r_0=0,r_1,\dots,r_n]$.
		\ENSURE The list of hook lengths of each rows of the associated Young diagram of the numerical set $R=\{ 0,r_1,\dots,r_n,\rightarrow \}$.
		\STATE \textbf{Initialize} $R=[0,r_1,\dots,r_n]$, $M=[\dots]$ and $f=r_n-1$
		\STATE Apply Algorithm \ref{alg1} to $R$ and find $P$
		\STATE Find $g=\len(P)$
		\STATE $M[1]=\ft(f - R, m \rightarrow m > 0)$
		\FOR{$i\in\{ 2,\dots,g\}$}
		\STATE $a=P[i-1]-P[i]+1$
		\STATE $M[i]=\ft(-a+M[i-1], m \rightarrow m > 0)$
		\ENDFOR
		\RETURN $M$
	\end{algorithmic}
\end{algorithm}

Algorithm \ref{alg3} takes the small elements of a numerical set $R$ and returns a list of lists where each inner list represents the hook lengths of boxes in a row of the Young diagrams $Y_R$. For example, if $R=\{ 0,4,6,7,9,10,12,\rightarrow \}$, the output of Algorithm \ref{alg3} for $R$ is $[[11,7,5,4,2,1],[8,4,2,1],[5,1],[1],[1],[1]]$. In this case, to draw the associated Young diagram $Y_R$ filled with hook lengths in LaTeX, one may use the \texttt{ytableau} package with \verb|\ytableaushort{{...}*{...}}| command.

For the algorithms above and the necessary calculations, GAP functions were produced and do not require any external GAP packages to run. However, when working specifically with numerical semigroups, the use of the \texttt{NumericalSgps} package \cite{DS} or \texttt{IntPic} package \cite{Del} will be needed.

\section{Operations on Young Diagrams}\label{sec4}
In this section, we utilize two operations introduced in Section 3 of \cite{SY} and introduce a new operation on Young diagrams. We also establish the corresponding interpretations of these operations in terms of numerical sets. Unless otherwise stated, all numerical sets considered in this section are assumed to be proper.

By analysing the construction of a Young diagram associated with a numerical set $R$, we observe that, for certain consecutive integers less than the Frobenius of $R$, there are only three different situations that can occur:
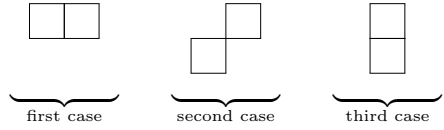
\vspace{1em}
\begin{center}
	\tiny{\begin{tabular}{ccccc}
			\ydiagram{2}& &\ydiagram{1+1,1}& &\ydiagram{1,1}\\
			$\underbrace{\hspace{4.8em}}_{\text{first case}}$&&$\underbrace{\hspace{4.8em}}_{\text{second case}}$&&$\underbrace{\hspace{4.8em}}_{\text{third case}}$
	\end{tabular}}
\end{center}
\vspace{1em}

The first case arises when there are no gaps between two consecutive small elements of $R$; the second case occurs when there is exactly one gap between two such elements; and the third case occurs when there are two or more gaps between them. These distinct local behaviours naturally motivate the definition of the following operations on Young diagrams.

\begin{definition}
	Let $Y$ be a Young diagram with $n$ columns and $g$ rows, and let $X$ be a Young diagram with $m$ columns and $k$ rows.
	\begin{enumerate}
		\item The \textit{bonded sum} of $Y$ and $X$, $Y\boxplus_{B}X$, is defined by glueing $X$ above $Y$ in such a way that the bottom row of $X$ is placed to right of the top row of $Y$. The resulting diagram is a Young diagram with $n+m$ columns and $g+k-1$ rows.
		\item The \textit{end-to-end sum} of $Y$ and $X$, $Y\boxplus_{E}X$, is defined by glueing $X$ above $Y$ such that the top-right corner of $Y$ aligns with the bottom-left corner of $X$. The resulting diagram is a Young diagram with $n+m$ columns and $g+k$ rows.
		\item The \textit{conjoint sum} of $Y$ and $X$, $Y\boxplus_{C}X$, is defined by glueing $X$ above $Y$ such that the most-left column of $X$ is placed directly on top of the right-most column of $Y$. The resulting diagram is a Young diagram with $n+m-1$ columns and $g+k$ rows.
	\end{enumerate}
\end{definition}

\begin{exam} Let $Y=[4,2,1]$ and $X=[4,3,1]$ be Young diagrams. We illustrate each of the three sum operations of $Y$ and $X$ as follows.
	\begin{enumerate}
		\item First, we draw $Y\boxplus_{B}X$
		
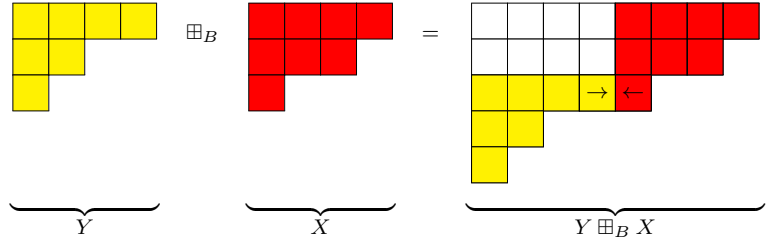
\vspace{1em}
		\begin{center}
			\tiny{\begin{tabular}{ccccc}
					\ydiagram[*(yellow) ]{4,2,1}&$\boxplus_{B}$&\ydiagram[*(red) ]{4,3,1}&$=$&\ydiagram[*(yellow) \rightarrow]
					{0,0,3+1,0,0}
					*[*(red) \leftarrow]
					{0,0,4+1,0,0}
					*[*(red) ]{4+4,4+3,4+1}*[*(yellow) ]{0,0,4,2,1}*[*(white) ]{8,7,5,2,1}\\
					$\underbrace{\hspace{6.4em}}_{\displaystyle Y}$&&$\underbrace{\hspace{6.4em}}_{\displaystyle X}$&&$\underbrace{\hspace{12.8em}}_{\displaystyle Y\boxplus_{B}X}$
			\end{tabular}}
		\end{center}
		\vspace{1em}
		\item Next, we illustrate $Y\boxplus_{E}X$
		
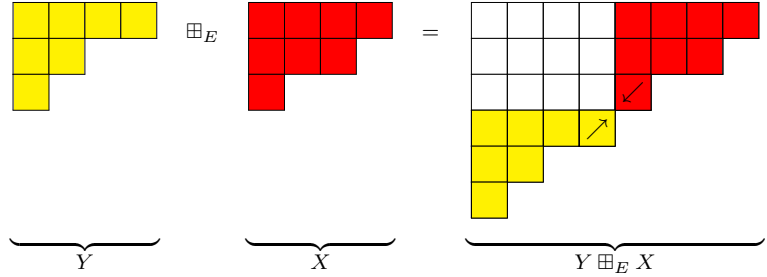
\vspace{1em}
		\begin{center}
			\tiny{\begin{tabular}{ccccc}
					\ydiagram[*(yellow) ]{4,2,1}&$\boxplus_{E}$&\ydiagram[*(red) ]{4,3,1}&$=$&\ydiagram[*(yellow) \nearrow]
					{0,0,0,3+1,0,0}
					*[*(red) \swarrow]
					{0,0,4+1,0,0,0}
					*[*(red) ]{4+4,4+3,4+1}*[*(yellow) ]{0,0,0,4,2,1}*[*(white) ]{8,7,5,4,2,1}\\
					$\underbrace{\hspace{6.4em}}_{\displaystyle Y}$&&$\underbrace{\hspace{6.4em}}_{\displaystyle X}$&&$\underbrace{\hspace{12.8em}}_{\displaystyle Y\boxplus_{E}X}$
			\end{tabular}}
		\end{center}
		\vspace{1em}
		\item Finally, we picture $Y\boxplus_{C}X$
		
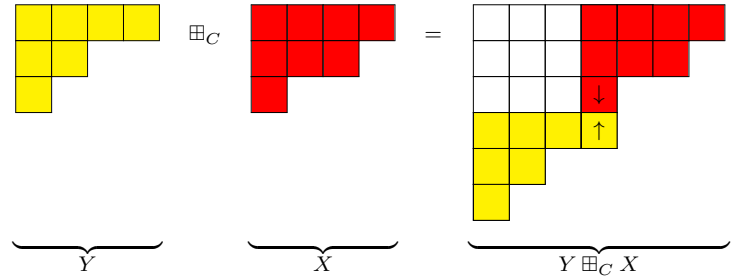
\vspace{1em}
		\begin{center}
			\tiny{\begin{tabular}{ccccc}
					\ydiagram[*(yellow) ]{4,2,1}&$\boxplus_{C}$&\ydiagram[*(red) ]{4,3,1}&$=$&\ydiagram[*(yellow) \uparrow]
					{0,0,0,3+1,0,0}
					*[*(red) \downarrow]
					{0,0,3+1,0,0,0}
					*[*(red) ]{3+4,3+3,3+1}*[*(yellow) ]{0,0,0,4,2,1}*[*(white) ]{5,3,3,3,2,1}\\
					$\underbrace{\hspace{6.4em}}_{\displaystyle Y}$&&$\underbrace{\hspace{6.4em}}_{\displaystyle X}$&&$\underbrace{\hspace{11.2em}}_{\displaystyle Y\boxplus_{C}X}$
			\end{tabular}}
		\end{center}
		\vspace{1em}
	\end{enumerate}
\end{exam}

By virtue of the bijection between the set of Young diagrams and the set of integer partitions, the sum operations introduced for Young diagrams naturally extend to partitions. In the following, we define the corresponding operations on the set of partitions.

\begin{definition}
	Let $Y=[\nu_1,\dots,\nu_n]$ and $X=[\mu_1,\dots,\mu_m]$ be Young diagrams. Then we define
	\begin{enumerate}
		\item $Y\boxplus_{B}X=[\nu_1+\mu_1,\dots,\nu_1+\mu_m,\nu_2,\dots,\nu_n]$,
		\item $Y\boxplus_{E}X=[\nu_1+\mu_1,\dots,\nu_1+\mu_m,\nu_1,\dots,\nu_n]$,
		\item $Y\boxplus_{C}X=[\nu_1+\mu_1-1,\dots,\nu_1+\mu_m-1,\nu_1,\dots,\nu_n]$.
	\end{enumerate}
\end{definition}

Due to the bijective correspondence between the set of Young diagrams and the set of proper numerical sets, the bonded, end-to-end, and conjoint sums can also be defined for numerical sets. Given two numerical sets $R$ and $S$ with associated Young diagrams $Y_{R}$ and $Y_{S}$, respectively, the bonded, end-to-end, and conjoint sums of $R$ and $S$ are defined as the numerical sets corresponding to the bonded, end-to-end, and conjoint sums of $Y_{R}$ and $Y_{S}$, respectively. 

In particular, when performing bonded sum of $Y_{R}$ and $Y_{S}$, the diagram $Y_{S}$ is vertically shifted by the number of rows in $Y_{R}$ minus one, and horizontally shifted by the number of columns in $Y_{R}$. Effectively, this displaces $Y_{S}$ exactly by the hook length of the top-left box in $Y_{R}$. 

Following this construction, and examining the resulting polygonal path, we observe that $\co(R)$ has been replaced by $\F(R)$, the zero element in $S$ has been eliminated, and each non-zero small element of $S$ is raised by $\F(R)$. Consequently, the small elements of the numerical set associated to the Young diagram $Y_{R}\boxplus_{B}Y_{S}$ are made up of the left elements of $R$, the Frobenius $\F(R)$, and the non-zero small elements of $S$ incremented by $F(R)$.

Over and above that we get the conductor of the resulting numerical set $R\boxplus_{C}S$ as $\co(R\boxplus_{C}S)=\co(R)+\co(S)-1$, and the Frobenius of $R\boxplus_{C}S$ as $\F(R\boxplus_{C}S)=\F(R)+\F(S)$. By analogous descriptions conjoint and end-to-end sums of $R$ and $S$ can be formulated as well. 

\begin{definition}
	Let $R=\{0,r_{1},\dots,r_{n}=\co(R),\rightarrow\}$ and $S=\{0,s_{1},\dots,s_{m}=\co(S),\rightarrow\}$ be numerical sets. Then we define
	\begin{enumerate}
		\item $R\boxplus_{B}S=\{0,r_{1},\dots,r_{n-1},r_{n}-1,s_{1}+r_{n}-1,\dots,s_{m}+r_{n}-1,\rightarrow\},$
		\item $R\boxplus_{E}S=\{0,r_{1},\dots,r_{n},s_{1}+r_{n},\dots,s_{m}+r_{n},\rightarrow\},$
		\item $R\boxplus_{C}S=\{0,r_{1},\dots,r_{n-1},s_{1}+r_{n}-1,\dots,s_{m}+r_{n}-1,\rightarrow\}.$
	\end{enumerate}
\end{definition}

The following lemma states the key properties of the sum operations on numerical sets defined above. The proofs are straightforward and follow directly from the definitions.

\begin{lemma} \label{l2}
	Let $R=\{0,r_{1},\dots,r_{n},\rightarrow\}$ and $S=\{0,s_{1},\dots,s_{m},\rightarrow\}$ be numerical sets with $\G(R)=\{ a_{1},\dots,a_{g} \}$ and $\G(S)=\{ b_{1},\dots,b_{k} \}$, respectively. Then
	\begin{enumerate}
		\item $\G(R\boxplus_{B}S)=\{ a_{1},\dots,a_{g-1},a_{g}+b_{1},\dots,a_{g}+b_{k} \}$ and $\F(R\boxplus_{B}S)=a_{g}+b_{k}$,
		\item $\G(R\boxplus_{E}S)=\{ a_{1},\dots,a_{g},a_{g}+b_{1}+1,\dots,a_{g}+b_{k}+1 \}$ and $\F(R\boxplus_{E}S)=a_{g}+b_{k}+1$,
		\item $\G(R\boxplus_{C}S)=\{ a_{1},\dots,a_{g},a_{g}+b_{1},\dots,a_{g}+b_{k} \}$ and $\F(R\boxplus_{C}S)=a_{g}+b_{k}$.
	\end{enumerate} 
\end{lemma}

\begin{remark}
	Let $R=\{0,r_{1},\dots,r_{n},\rightarrow\}$ be a numerical set. Then $\N_0 \boxplus_{B} R=\{ -1,0,r_1-1,\dots,r_n-1,\rightarrow\}$ is not a numerical set, and $\N_0 \boxplus_{C} R=\{ r_1-1,\dots,r_n-1,\rightarrow\}$ is not also a numerical set unless $r_1=1$. However, we naturally have the following.
	\begin{align*}
		R\boxplus_{B}\N_0=&\{0,r_1,\dots,r_{n-1},r_n-1,\rightarrow\},&R\boxplus_{C}\N_0=&\{0,r_1,\dots,r_{n-1},r_n-1,\rightarrow\},
		\\
		R\boxplus_{E}\N_0=&\{0,r_1,\dots,r_n,\rightarrow\},&
		\N_0 \boxplus_{E} R=&\{0,r_1,\dots,r_n,\rightarrow\}.
	\end{align*}
\end{remark}

Notice that $R\boxplus_{E}S=R\boxplus_{C}\{ 0,2,\rightarrow\}\boxplus_{B}S$ for all proper numerical sets $R$ and $S$. Moreover, it is straightforward to verify that the bonded, end-to-end, and conjoint sums are binary operations on the set of proper numerical sets that are closed, associative, and non-commutative. These structural properties lead directly to the following proposition, whose proof follows immediately from the definitions.

\begin{proposition}
	Let $\mathcal{R}$ be the set of all numerical sets and $\mathcal{R}_0$ the set of all proper numerical sets. Then
	\begin{enumerate}
		\item $(\mathcal{R}_0,\boxplus_{B})$ and $(\mathcal{R}_0,\boxplus_{C})$ are semigroups,
		\item $(\mathcal{R},\boxplus_{E})$ is a monoid where $\N_0$ is the identity.
	\end{enumerate} 
\end{proposition}

Furthermore, those sums are not closed in the set of proper numerical semigroups. For example, although  $\{0,2,\rightarrow\}$, $\{0,3,\rightarrow\}$ and $\{0,2,4,\rightarrow\}$ are numerical semigroups, $$\{0,2,4,\rightarrow\}\boxplus_{B}\{0,2,\rightarrow\}=\{0,2,3,5,\rightarrow\},$$ $$\{0,2,\rightarrow\}\boxplus_{E}\{0,3,\rightarrow\}=\{0,2,5,\rightarrow\},$$ $$\{0,2,4,\rightarrow\}\boxplus_{C}\{0,2,\rightarrow\}=\{0,2,5,\rightarrow\}$$ are not numerical semigroups.

\begin{proposition} \label{prp2}
	Let $R$ and $S$ be numerical sets. Then
	\begin{enumerate}
		\item $(R\boxplus_{B}S)^{*}=S^{*}\boxplus_{C}R^{*}$,
		\item $(R\boxplus_{E}S)^{*}=S^{*}\boxplus_{E}R^{*}$,
		\item $(R\boxplus_{C}S)^{*}=S^{*}\boxplus_{B}R^{*}$.
	\end{enumerate} 
\end{proposition}

\begin{proof}
	Let $R=\{ 0,r_1,\dots,r_n,\rightarrow\}$ with $\G(R)=\{ a_{1},\dots,a_{g} \}$ and $S=\{ 0,s_1,\dots,s_m,\rightarrow \}$ with $\G(S)=\{ b_{1},\dots,b_{k} \}$. Then 
	\begin{align*}
		R\boxplus_{B}S=&\{0,r_{1},\dots,r_{n-1},r_{n}-1,s_{1}+r_{n}-1,\dots,s_{m}+r_{n}-1,\rightarrow\}\\
		=&\{0,r_{1},\dots,r_{n-1},a_g,a_g+s_{1},\dots,a_g+s_{m},\rightarrow\}
	\end{align*}
	with $\G(R\boxplus_{B}S)=\{ a_{1},\dots,a_{g-1},a_{g}+b_{1},\dots,a_{g}+b_{k} \}$ and $f=\F(R\boxplus_{B}S)=a_g+b_k$. Hence,
	\begin{align*}
		(R\boxplus_{B}S)^{*}=&\{0,f-a_g-b_{k-1},\dots,f-a_g-b_1,f-a_{g-1},\dots,f-a_1,a_g+s_m,\rightarrow\}\\
		=&\{0,b_k-b_{k-1},\dots,b_k-b_1,b_k+a_g-a_{g-1},\dots,b_k+a_g-a_1,a_g+s_m,\rightarrow\}\\
		=&\{0,b_k-b_{k-1},\dots,b_k-b_1,s_m,\rightarrow\}\boxplus_{C}\{0,a_g-a_{g-1},\dots,a_g-a_1,a_g+1,\rightarrow\}\\
		=&S^{*}\boxplus_{C}R^{*}.
	\end{align*}
	This shows that the first statement holds. The other statements follow in similar ways.
\end{proof}

\section{Decompositions of Almost Symmetric Numerical Semigroups}\label{sec5}

In this section, we broaden the scope of the Young diagram decompositions previously developed for irreducible numerical semigroups in \cite{SY}, extending them to a particular class of almost symmetric numerical semigroups. In particular, we focus on almost symmetric semigroups whose pseudo-Frobenius numbers form a consecutive sequence, excluding the Frobenius number itself. We introduce new decompositions for this subclass using the combinatorial structure of their associated Young diagrams.

\begin{theorem} \label{thm1}
	Let $R$ be a symmetric numerical set and $R\neq\{0,2,\rightarrow\}$. Then there exist a unique numerical set $S$ such that $$R=S\boxplus_{C}\{0,2,\rightarrow\}\boxplus_{B}S^{*} \text{ or } R=S\boxplus_{B}\{0,2,\rightarrow\}\boxplus_{C}S^{*}.$$
\end{theorem}

\begin{proof} Let $R=\{ 0,r_{1},\dots,r_{n}=\co(R),\rightarrow \}$ with $r_{k}\leq\frac{\co(R)}{2}<r_{k+1}$ for a small element $r_k$ of $R$ and $\G(R)=\{ a_{1},\dots,a_{n-1}, a_{n}=\F(R) \}$. Since $R$ is symmetric, for each gap $a_i$, we know that $\F(R)-a_i\in R$. This implies that $R=\{ 0,\F(R)-a_{n-1},\dots,\F(R)-a_{1},r_{n},\rightarrow \}$ where $r_{i}=\F(R)-a_{r-i}$ for $i=1,2,\dots,n-1$. We have two cases to examine:
	
	\paragraph{Case 1:} In case $\frac{\co(R)}{2}\in R$, we choose $S=\{ 0,r_{1},\dots,r_{k},\rightarrow \}$ with $\co(S)=r_k=\frac{\co(R)}{2}$ and $f=\F(S)=\frac{\co(R)}{2}-1$. Then we get  $\G(S)=\{ a_{1},\dots,a_{n-k} \}$ and $$S^{*}=\{0, f-a_{n-k-1},\dots,f-a_{1},r_k,\rightarrow \}.$$ Thus,
	\begin{align*}
		S\boxplus_{C}\{0,2,\rightarrow\}\boxplus_{B}S^{*}=&\{ 0,r_{1},\dots,r_{k},\rightarrow \}\boxplus_{C}\{0,2,\rightarrow\}\boxplus_{B}\\
		&\hspace{2cm}\{0, f-a_{n-k-1},\dots,f-a_{1},r_k,\rightarrow \}\\
		=&\{ 0,r_{1},\dots,r_k+1,\rightarrow \}\boxplus_{B}\\
		&\hspace{2cm}\{0,f-a_{n-k-1},\dots,f-a_{1},r_k,\rightarrow \}\\
		=&\{ 0,r_{1},\dots,r_{k-1},r_{k}+1-1,r_{k}+1+f-a_{n-k-1}-1,\dots\\
		&\hspace{2cm}\dots,r_k+1+f-a_{1}-1,r_k+1+r_k-1, \rightarrow \}\\
		=&\{ 0,r_{1},\dots,r_{k},r_k+f-a_{n-k-1},\dots,r_k+f-a_{1},2r_k, \rightarrow \}\\
		=&\{ 0,r_{1},\dots,r_{k},\F(R)-a_{n-k-1},\dots,\F(R)-a_{1},\co(R), \rightarrow \}\\
		=&\{ 0,r_{1},\dots,r_{k},r_{k+1},\dots,r_{n-1},r_{n}, \rightarrow \}=R.
	\end{align*}
	
	\paragraph{Case 2:} In case $\frac{\co(R)}{2}\notin R$, we select $S=\{ 0,r_{1},\dots,r_{k-1},r_k+1,\rightarrow \}$ with the gap set $\G(S)=\{ a_{1},\dots,a_{n-k},r_k \}$. Then we get $\F(R)-\frac{\co(R)}{2}\in R \implies \frac{\co(R)}{2}-1\in R\implies \F(S)=r_{k}=\frac{\co(R)}{2}-1$ and  $\co(S)=r_k+1=\frac{\co(R)}{2}$. The dual of $S$ is then $S^{*}=\{0, r_k-a_{n-k-1},\dots,r_k-a_{1},r_k+1,\rightarrow \}.$ Therefore,
	\begin{align*}
		S\boxplus_{B}\{0,2,\rightarrow\}\boxplus_{C}S^{*}=&\{ 0,r_{1},\dots,r_{k-1},r_k+1,\rightarrow \}\boxplus_{B}\{0,2,\rightarrow\}\boxplus_{C}\\
		&\hspace{2cm}\{0, r_k-a_{n-k-1},\dots,r_k-a_{1},r_k+1,\rightarrow \}\\
		=&\{ 0,r_{1},\dots,r_{k-1},r_k,r_k+2,\rightarrow \}\boxplus_{C}\\
		&\hspace{2cm}\{0, r_k-a_{n-k-1},\dots,r_k-a_{1},r_k+1,\rightarrow \}\\
		=&\{ 0,r_{1},\dots,r_{k-1},r_{k},r_{k}+2+r_k-a_{n-k-1}-1,\dots\\
		&\hspace{2cm}\dots,r_k+2+r_k-a_{1}-1,r_k+2+r_k+1-1, \rightarrow \}\\
		=&\{ 0,r_{1},\dots,r_{k-1},r_{k},2r_k+1-a_{n-k-1},\dots,2r_k+1-a_{1},2r_k+2, \rightarrow \}\\
		=&\{ 0,r_{1},\dots,r_{k},\F(R)-a_{n-k-1},\dots,\F(R)-a_{1},\co(R), \rightarrow \}\\
		=&\{ 0,r_{1},\dots,r_{k},r_{k+1},\dots,r_{n-1},r_{n}, \rightarrow \}=R.
	\end{align*}
\end{proof}

The converse of Theorem \ref{thm1} also holds. This means that for a given arbitrary numerical set, we can always construct a symmetric numerical set using its dual. So, the next is the symmetric construction lemma.

\begin{lemma}\label{sym}
	Let $S$ be a numerical set. Then the following are always symmetric:
	\begin{multicols}{2}
		\begin{enumerate}
			\item $S\boxplus_{C}\{0,2,\rightarrow\}\boxplus_{B}S^{*}$,
			\item $S\boxplus_{B}\{0,2,\rightarrow\}\boxplus_{C}S^{*}$,
			\item $S^{*}\boxplus_{C}\{0,2,\rightarrow\}\boxplus_{B}S$,
			\item $S^{*}\boxplus_{B}\{0,2,\rightarrow\}\boxplus_{C}S$.
		\end{enumerate}
	\end{multicols}
\end{lemma}

\begin{proof}
	By Proposition \ref{prp2}, since $\{0,2,\rightarrow\}$ is symmetric we get
	\begin{align*}
		(S\boxplus_{C}\{0,2,\rightarrow\}\boxplus_{B}S^{*})^{*}= &(S^{*})^{*}\boxplus_{C}(S\boxplus_{C}\{0,2,\rightarrow\})^{*}\\ = &(S^{*})^{*}\boxplus_{C}(\{0,2,\rightarrow\})^{*}\boxplus_{B}S^{*}\\ = &S\boxplus_{C}\{0,2,\rightarrow\}\boxplus_{B}S^{*},\\
	\end{align*}
	and so $S\boxplus_{C}\{0,2,\rightarrow\})\boxplus_{B}S^{*}$ is symmetric. The proof follows for the others in similar ways.
\end{proof}

\begin{exam}
	Let $R=\{ 0,2,5,6,8,10,\rightarrow \}$ with $\F(R)=9$ and $\G(R)=\{ 1,3,4,7,9 \}$. It is symmetric and $\frac{\co(R)}{2}=5\in R$. So, by Theorem \ref{thm1}, $S=\{ 0,2,5,\rightarrow \}$ with $\G(S)=\{ 1,3,4 \}$ and $S^{*}=\{ 0,1,3,5,\rightarrow \}$ where we get $R=S\boxplus_{C}\{0,2,\rightarrow\}\boxplus_{B}S^{*}$. Then the Young diagram $Y_{R}=Y_{S}\boxplus_{C}Y_{\{0,2,\rightarrow\}}\boxplus_{B}Y_{S^{*}}$ is illustrated as follows.
	\vspace{1em}
	\begin{center}
		\tiny{\begin{tabular}{ccccccc}
				\ydiagram[*(yellow) ]{2,2,1}&$\boxplus_{C}$&\ydiagram[*(orange) ]{1}&$\boxplus_{B}$&\ydiagram[*(red) ]{3,2}&$=$&\ydiagram[*(red) ]{2+3,2+2}*[*(orange) ]{0,1+1}*[*(yellow) ]{,0,0,2,2,1,}*[*(white) ]{5,4,2,2,1}\\
				$\underbrace{\hspace{3.2em}}_{\displaystyle Y_{S}}$&&$\underbrace{\hspace{1.6em}}_{\displaystyle Y_{\{0,2,\rightarrow\}}}$&&$\underbrace{\hspace{4.8em}}_{\displaystyle Y_{S^{*}}}$&&$\underbrace{\hspace{8em}}_{\displaystyle Y_{R}=Y_{S}\boxplus_{C}Y_{\{0,2,\rightarrow\}}\boxplus_{B}Y_{S^{*}}}$
		\end{tabular}}
	\end{center}
	\vspace{1em}
\end{exam}

\begin{exam}
	Let $R=\{ 0,2,5,7,8,10,12\rightarrow \}$ with $\G(R)=\{ 1,3,4,6,9,11 \}$ and $\F(R)=11$. It is symmetric and $\frac{\co(R)}{2}=6\notin R$. Therefore, by Theorem \ref{thm1}, $S=\{ 0,2,6,\rightarrow \}$ with $\G(S)=\{ 1,3,4,5 \}$ and $S^{*}=\{ 0,1,2,4,6,\rightarrow \}$ where we get $R=S\boxplus_{B}\{0,2,\rightarrow\}\boxplus_{C}S^{*}$. Then the Young diagram $Y_{R}=Y_{S}\boxplus_{B}Y_{\{0,2,\rightarrow\}}\boxplus_{C}Y_{S^{*}}$ is illustrated as follows.
	\vspace{1em}
	\begin{center}
		\tiny{\begin{tabular}{ccccccc}
				\ydiagram[*(yellow) ]{2,2,2,1}&$\boxplus_{B}$&\ydiagram[*(orange) ]{1}&$\boxplus_{C}$&\ydiagram[*(red) ]{4,3}&$=$&\ydiagram[*(red) ]{2+4,2+3}*[*(orange) ]{0,0,2+1}*[*(yellow) ]{0,0,2,2,2,1,}*[*(white) ]{6,5,3,2,2,1}\\
				$\underbrace{\hspace{3.2em}}_{\displaystyle Y_{S}}$&&$\underbrace{\hspace{1.6em}}_{\displaystyle Y_{\{0,2,\rightarrow\}}}$ &&$\underbrace{\hspace{6.4em}}_{\displaystyle Y_{S^{*}}}$&&$\underbrace{\hspace{9.6em}}_{\displaystyle Y_{R}=Y_{S}\boxplus_{B}Y_{\{0,2,\rightarrow\}}\boxplus_{C}Y_{S^{*}}}$
		\end{tabular}}
	\end{center}
	\vspace{1em}
\end{exam}

When decomposing a symmetric numerical semigroup $R=\{0,r_1\dots,r_n,\rightarrow\}$ into a numerical semigroup and its dual, some certain considerations must be made. By Theorem \ref{thm1}, when $\frac{\co(R)}{2}\notin R$, the natural choice would be $S=\{ 0,r_1,\dots,r_{k-1},r_k+1,\rightarrow \}$ where $r_k=\frac{\co(R)}{2}-1$ for some $k\in[1,n]\cap\N_0$. In this case, if $r_k$ is not a minimal generator of $R$, then $S$ fails to be a numerical semigroup. Thus, to ensure $S$ remains a numerical semigroup, we need to consider a slightly different $S$ and a different sum.

Hence, when $R=\{0,r_1\dots,r_n,\rightarrow\}$ is a symmetric numerical semigroup with $\frac{\co(R)}{2}\notin R$, to ensure $S$ remains a numerical semigroup in the decomposition, we need to choose that $S=\{ 0,r_1,\dots,r_{k},r_k+2\rightarrow \}$ with $r_k=\frac{\co(R)}{2}-1$ for some $k\in[1,n]\cap\N_0$. In this way, we guarantee that $S$ is a numerical semigroup. However, the decomposition takes the form $R=S\boxplus_{C}(S^{*}-1)$ where $S^{*}-1=\{ s-1 \mid s\in S^{*}\cap\N \}\cup\{0\}$. To prove this statement, we shall use a similar way as in the proof of Theorem \ref{thm1}.

Additionally, since $S\boxplus_{C}\{0,2,\rightarrow\}\boxplus_{B}T=S\boxplus_{E}T$ for all proper numerical sets $S$ and $T$, in case $\frac{\co(R)}{2}\in R$, we can equivalently consider a decomposition $R=S\boxplus_{E}S^{*}$. Thus, we get the following.

\begin{corollary} \cite[cf. Theorem 4.4]{SY}
	Let $R$ be a symmetric numerical semigroup. Then there exists a unique numerical semigroup $S$ such that $$R=S\boxplus_{C}\{0,2,\rightarrow\}\boxplus_{B}S^{*}=S\boxplus_{E}S^{*} \text{ or } R=S\boxplus_{C}(S^{*}-1).$$
\end{corollary}

\begin{remark} \label{rmk}
	Let $R$ be a numerical semigroup with type $\ty(R)=t$, and let $\pf(R)=\{ a_1,\dots,a_t=\F(R) \}$ where $a_i<a_{i+1}$ for $i=1,\dots,t-1$. By Theorem 2.4 in \cite{nari}, $R$ is almost symmetric if and only if $a_i+a_{t-i}=a_t$ for all $i\in [1,t-1]\cap\N_0$.
\end{remark}

\begin{lemma} \label{lem}
	Let $R=\{ 0,r_1,\dots,r_n,\rightarrow\}$ be an almost symmetric numerical semigroup with type $t\geq2$, and let $\pf(R)=\{ a_1,\dots,a_t=\F(R) \}$. Suppose $r_k<a_1<\dots<a_{t-1}<r_{k+1}$ where $a_i-a_{i-1}=1$ for $i=2,\dots,t-1$. Then $R'=\{ 0,r_1,\dots,r_k,r_{k+1}',\dots,r_n',\rightarrow \}$, where $r_{k+i}'=r_{k+i}-t+1$ for $i=1,\dots,n-k$, is a symmetric numerical set.
\end{lemma}

\begin{proof}
	By Remark \ref{rmk}, we have $a_1<\frac{\F(R)}{2}<a_{t-1}$. Let $a\in[0,\F(R')]$ and suppose $a\in\Z\setminus R'$. Then $a\in[0,a_1)$ or $a\in[a_1,\F(R')]$.
	
	\paragraph{Case 1:} If $a\in[0,a_1)$, then $a\in\Z\setminus R'$ implies $a\in\Z\setminus R$ and $a\notin \pf(R)$. Since $R$ is almost symmetric, $\F(R)-a\in R$, and thus $\F(R)-a=r_{j}$ for some $j\in[k+1,n-1]\cap\N$. This means $\F(R)-a-t+1=r_j-t+1 \in R'$, which implies $\F(R')-a=r_j'\in R'$.
	
	\paragraph{Case 2:} If $a\in[a_1,\F(R')]$, then $a+t-1\in [a_1+t-1,\F(R')+t-1]$, and so $a+t-1\in(a_{t-1},\F(R)]$. Thus, $a+t-1\in\Z\setminus R$ and $a+t-1\notin \pf(R)$. Since $R$ is almost symmetric, $\F(R)-a-t+1\in R$, and $\F(R)-a-t+1=r_j$ for some $j\in[1,k]\cap\N$. This means $\F(R)-t+1-a=\F(R')-a=r_j$.
	
	Hence, $R'$ is a symmetric numerical set.
\end{proof}

\begin{theorem} \label{mainthm}
	Let $R=\{ 0,r_1,\dots,r_n,\rightarrow\}$ be an almost symmetric numerical semigroup with type $t\geq2$ and $R\neq\{0,t+1,\rightarrow\}$, and let $\pf(R)=\{ a_1,\dots,a_t=\F(R) \}$. Suppose $r_{k-1}<a_1<\dots<a_{t-1}<r_k$ where $a_i-a_{i-1}=1$ for $i=2,\dots,t-1$. Then there is a unique numerical set $S$ such that $$R=S\boxplus_{C}\{0,t+1,\rightarrow\}\boxplus_{B}S^{*} \text{ or } R=S\boxplus_{B}\{0,t+1,\rightarrow\}\boxplus_{C}S^{*}.$$
\end{theorem}

\begin{proof}
	By Remark \ref{rmk}, if $t$ is odd, we have $a_{\frac{t-1}{2}} < \frac{\F(R)}{2} < \frac{\co(R)}{2}=a_{\frac{t+1}{2}}$, and if $t$ is even, we have $a_{\frac{t}{2}} = \frac{\F(R)}{2} < \frac{\co(R)}{2}<a_{\frac{t}{2}+1}$. This means $r_{k-1}<\frac{\co(R)}{2}<r_k$ where $\frac{\co(R)}{2}-r_{k-1}>\frac{t-1}{2}$ and $r_k-\frac{\co(R)}{2}\geq\frac{t-1}{2}$.
	
	Let $R'=\{ 0,r_1,\dots,r_{k-1},r_k',\dots,r_n',\rightarrow \}$, where $r_{k+i}'=r_{k+i}-t+1$ for $i=0,\dots,n-k$. By Lemma \ref{lem}, we know that $R'$ is symmetric. On the other hand, we have $\co(R')=\co(R)-t+1$ and $\F(R')=\co(R)-t$ which implies  $r_{k-1}<\frac{\co(R')}{2}\leq r_k'$. Then we get two cases to consider.
	
	\paragraph{Case 1:} In case $\frac{\co(R')}{2}\in R'$, we get $r_k'=\frac{\co(R')}{2}=\frac{\co(R)-t+1}{2}$ and we choose $$S=\{ 0,r_1,\dots,r_{k-1},r_k',\rightarrow \}.$$ Since $R'$ is symmetric, $\g(S)=n-k$. Then let $\G(S)=\{b_1,\dots,b_{n-k}\}$, $f=\F(S)=b_{n-k}=\frac{\co(R)-t-1}{2}$ and $$S^{*}=\{ 0,f-b_{n-k-1}\dots,f-b_1,r_k',\rightarrow\},$$
	where $\F(R')-b_i=r_{n-i}'$ for $i=1,\dots,n-k$. Therefore,
	\begin{align*}
		S\boxplus_{C}\{0,t+1,\rightarrow\}\boxplus_{B}S^{*}=&\{ 0,r_{1},\dots,r_{k}',\rightarrow \}\boxplus_{C}\{0,t+1,\rightarrow\}\boxplus_{B}\\ &\hspace{2.4cm}\{0, f-b_{n-k-1},\dots,f-b_{1},r_k',\rightarrow \}\\
		=&\{ 0,r_{1},\dots,r_{k-1},r_k'+t,\rightarrow \}\boxplus_{B}\{0, f-b_{n-k-1},\dots,f-b_{1},r_k',\rightarrow \}\\
		=&\{ 0,r_{1},\dots,r_{k-1},r_k'+t-1,r_k'+t-1+f-b_{n-k-1},\dots\\
		&\hspace{2.4cm}\dots,r_k'+t-1+f-b_{1},r_k'+t-1+r_k', \rightarrow \}.
	\end{align*}
	Hence, since $r_k'=r_k-t+1$ and $r_k'+f=\co(R)-t=\F(R')$, we get the following
	\begin{align*}
		S\boxplus_{C}\{0,t+1,\rightarrow\}\boxplus_{B}S^{*}=&\{ 0,r_{1},\dots,r_{k},\F(R')-b_{n-k-1}+t-1,\dots\\ &\hspace{2.4cm}\dots,\F(R')-b_{1}+t-1,\co(R), \rightarrow \}\\
		=&\{ 0,r_{1},\dots,r_{k},r_{k+1}'+t-1,\dots,r_{n-1}'+t-1,\co(R), \rightarrow \}\\
		=&\{ 0,r_{1},\dots,r_{k},r_{k+1},\dots,r_{n-1},r_{n}, \rightarrow \}=R.
	\end{align*}
	
	\paragraph{Case 2:} In case $\frac{\co(R')}{2}\notin R'$, we choose $S=\{ 0,r_1,\dots,r_{k-2},r_{k-1}+1,\rightarrow \}$ with $\G(S)=\{ b_1,\dots,b_{n-k},r_{k-1} \}$. Then we get $\F(R')-\frac{\co(R')}{2}=\frac{\co(R')}{2}-1\in R'$, and so $$f=\F(S)=r_{k-1}=\frac{\co(R')}{2}-1 \text{ and } \co(S)=\frac{\co(R')}{2}.$$ We also get $$S^{*}=\{ 0,f-b_{n-k}\dots,f-b_1,\frac{\co(R')}{2},\rightarrow\}.$$
	Therefore,
	\begin{align*}
		S\boxplus_{B}\{0,t+1,\rightarrow\}\boxplus_{C}S^{*}=&\{ 0,r_{1},\dots,r_{k-2},r_{k-1}+1,\rightarrow \}\boxplus_{B}\{0,t+1,\rightarrow\}\boxplus_{C}\\
		&\hspace{3cm}\{0,f-b_{n-k},\dots,f-b_{1},\frac{\co(R')}{2},\rightarrow \}\\
		=&\{ 0,r_{1},\dots,r_{k-2},r_{k-1},\frac{\co(R')}{2}+t,\rightarrow \}\boxplus_{C}\\
		&\hspace{2cm}\{0, \frac{\co(R')}{2}-1-b_{n-k},\dots,\frac{\co(R')}{2}-1-b_{1},\frac{\co(R')}{2},\rightarrow \}\\
		=&\{ 0,r_{1},\dots,r_{k-1},\co(R')-1-b_{n-k}+t-1,\dots\\
		&\hspace{3cm}\dots,\co(R')-1-b_1+t-1,\co(R')+t-1, \rightarrow \}.
	\end{align*}
	
	Since $R'$ is symmetric, we get $\co(R')-1-b_j=\F(R')-b_j=r_{n-j}'$ for $j=1,\dots,n-k$. Then
	
	\begin{align*}
		S\boxplus_{B}\{0,t+1,\rightarrow\}\boxplus_{C}S^{*}=&\{ 0,r_{1},\dots,r_{k-1},r_k'+t-1,\dots,r_{n-1}'+t-1,r_n'+t-1, \rightarrow \}\\=&\{0,r_1,\dots,r_k,\dots,r_n,\rightarrow\}=R.
	\end{align*}
	
\end{proof}

\begin{exam}
	Let $R=\{0,4,8,12,14,15,16,18,19,20,22,\rightarrow \}$. It is an almost symmetric numerical semigroup with $\pf(R)=\{10,11,21\}$ and $\ty(R)=3$. We find $$R'=\{ 0,4,8,10,12,13,14,16,17,18,20,\rightarrow\}$$ and we get $\frac{\co(R')}{2}=10\in R'$. Therefore, by Theorem \ref{mainthm}, $S=\{ 0,4,8,10,\rightarrow \}$ where we have $\G(S)=\{ 1,2,3,5,6,7,9 \}$ and $S^{*}=\{ 0,2,3,4,6,7,8,10,\rightarrow\}$. Hence, we get $$R=S\boxplus_{C}\{0,4,\rightarrow\}\boxplus_{B}S^{*},$$
	and the Young diagram of $Y_R=Y_S\boxplus_{C}Y_{\{0,4,\rightarrow\}}\boxplus_{B}Y_{S^{*}}$ is depicted below
	
	
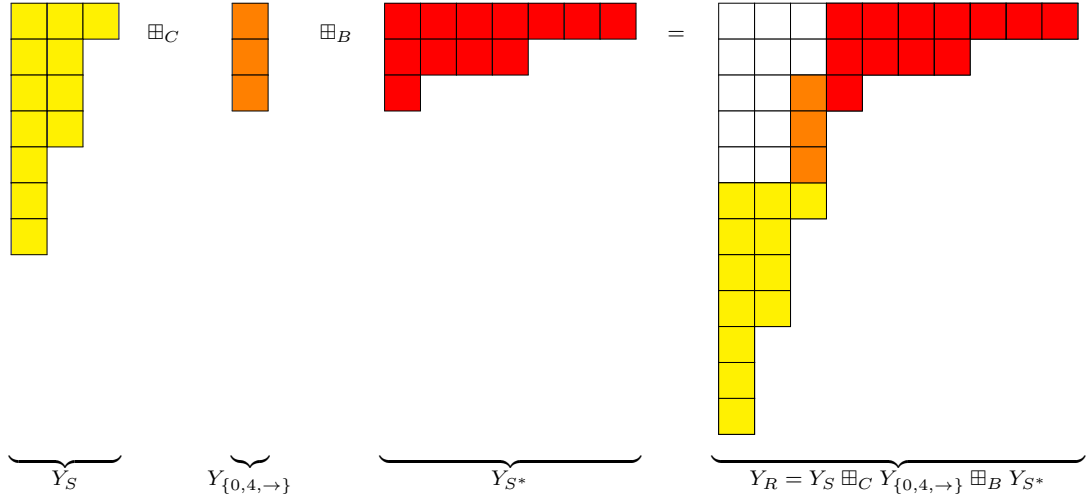
\vspace{1em}
	\begin{center}
		\tiny{\begin{tabular}{ccccccc}
				\ydiagram[*(yellow) ]{3,2,2,2,1,1,1}&$\boxplus_{C}$&\ydiagram[*(orange) ]{1,1,1}&$\boxplus_{B}$&\ydiagram[*(red) ]{7,4,1}&$=$&\ydiagram[*(red) ]{3+7,3+4,3+1}*[*(orange) ]{0,0,2+1,2+1,2+1}*[*(yellow) ]{0,0,0,0,0,3,2,2,2,1,1,1}*[*(white) ]{10,7,4,3,3,3,2,2,2,1,1,1}\\
				$\underbrace{\hspace{4.8em}}_{\displaystyle Y_{S}}$&&$\underbrace{\hspace{1.6em}}_{\displaystyle Y_{\{0,4,\rightarrow\}}}$&&$\underbrace{\hspace{11.2em}}_{\displaystyle Y_{S^{*}}}$&&$\underbrace{\hspace{16em}}_{\displaystyle Y_{R}=Y_{S}\boxplus_{C}Y_{\{0,4,\rightarrow\}}\boxplus_{B}Y_{S^{*}}}$
		\end{tabular}}
	\end{center}
	\vspace{1em}
	
\end{exam}

\begin{exam}
	Let $R=\{0,7,9,14,16,17,18,19,20,21,23,\rightarrow \}$. It is an almost symmetric numerical semigroup with $\PF(R)=\{ 10,11,12,22 \}$ and $\ty(R)=4$. We find $$R'=\{ 0, 7, 9, 11, 13, 14, 15, 16, 17, 18, 20,\rightarrow \}$$
	for which we have $\frac{\co(R')}{2}=10 \notin R'$. Thus, by Theorem \ref{mainthm}, we get $S=\{ 0,7,10,\rightarrow \}$, and so $\G(S)=\{ 1,2,3,4,5,6,8,9 \}$ and $S^{*}=\{ 0,1,3,4,5,6,7,8,10,\rightarrow \}$. Hence, we get $$R=S\boxplus_{B}\{0,5,\rightarrow\}\boxplus_{C}S^{*},$$
	and the Young diagram of $Y_R=Y_S\boxplus_{B}Y_{\{0,5,\rightarrow\}}\boxplus_{C}Y_{S^{*}}$ is depicted below
	
	
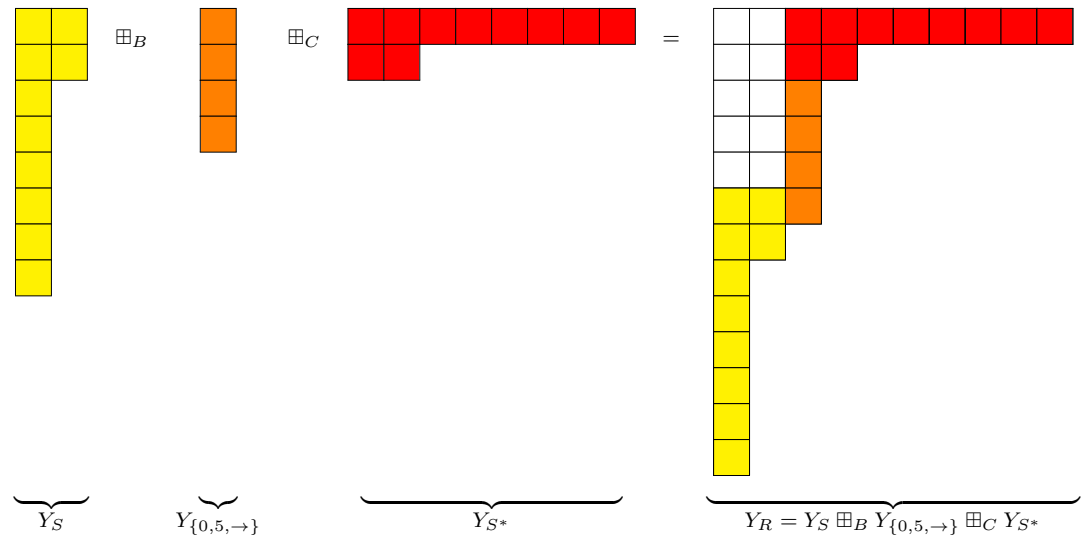
\vspace{1em}
	\begin{center}
		\tiny{\begin{tabular}{ccccccc}
				\ydiagram[*(yellow) ]{2,2,1,1,1,1,1,1}&$\boxplus_{B}$&\ydiagram[*(orange) ]{1,1,1,1}&$\boxplus_{C}$&\ydiagram[*(red) ]{8,2}&$=$&\ydiagram[*(red) ]{2+8,2+2}*[*(orange) ]{0,0,2+1,2+1,2+1,2+1}*[*(yellow) ]{0,0,0,0,0,2,2,1,1,1,1,1,1}*[*(white) ]{10,4,3,3,3,3,2,1,1,1,1,1,1}\\
				$\underbrace{\hspace{3.2em}}_{\displaystyle Y_{S}}$&&$\underbrace{\hspace{1.6em}}_{\displaystyle Y_{\{0,5,\rightarrow\}}}$&&$\underbrace{\hspace{11.2em}}_{\displaystyle Y_{S^{*}}}$&&$\underbrace{\hspace{16em}}_{\displaystyle Y_{R}=Y_{S}\boxplus_{B}Y_{\{0,5,\rightarrow\}}\boxplus_{C}Y_{S^{*}}}$
		\end{tabular}}
	\end{center}
	\vspace{1em}
\end{exam}

Let $R=\{0,r_1\dots,r_n,\rightarrow\}$ be an almost symmetric numerical semigroup with the notation and assumptions as in Theorem \ref{mainthm}. Similarly, to decompose $R$ into a numerical semigroup and its dual, we need to be careful. When $\frac{\co(R')}{2}\notin R'$, the naive choice $S=\{ 0,r_1,\dots,r_{k-2}, r_{k-1}+1,\rightarrow \}$ where $r_{k-1}=\frac{\co(R')}{2}-1$ may fail to be a numerical semigroup if $r_{k-1}$ is not a minimal generator for $S$. Thus, to guarantee $S$ is a numerical semigroup always, we instead choose $S=\{ 0,r_1,\dots,r_{k-1},r_{k-1}+2,\rightarrow \}$ where $r_{k-1}=\frac{\co(R')}{2}-1$ and consider a slightly different sum. However, the decomposition becomes as $R=S\boxplus_{C}S^{*}$ when $\ty(R)=2$, i.e. $R$ is pseudo-symmetric, and the decomposition becomes as $R=S\boxplus_{C}\{0,t-1,\rightarrow\}\boxplus_{C}S^{*}$ when $\ty(R)>2$. To prove this statement, we shall use a similar way as in the proof of Theorem \ref{mainthm}.

Furthermore, notice that for any positive integer $k>1$ and any numerical set $T$, we have $\{0,k+1,\rightarrow\}\boxplus_{B}T=\{0,k,\rightarrow\}\boxplus_{E}T$. Thus, we get the following corollaries.

\begin{corollary} \cite[cf. Theorem 4.10]{SY}
	Let $R=\{ 0,r_1,\dots,r_n,\rightarrow\}$ be a pseudo-symmetric numerical semigroup and $R\neq \{0,3,\rightarrow\}$. Then there is a unique numerical semigroup $S$ such that $$R=S\boxplus_{C}\{0,3,\rightarrow\}\boxplus_{B}S^{*}=S\boxplus_{C}\{0,2,\rightarrow\}\boxplus_{E}S^{*} \text{ or } R=S\boxplus_{C}S^{*}.$$	
\end{corollary}

\begin{corollary}
	Let $R=\{ 0,r_1,\dots,r_n,\rightarrow\}$ be an almost symmetric numerical semigroup with type $t>2$ not equal to $\{0,t+1,\rightarrow\}$, and let $\pf(R)=\{ a_1,\dots,a_t=\F(R) \}$. Suppose $r_{k-1}<a_1<\dots<a_{t-1}<r_k$ where $a_i-a_{i-1}=1$ for $i=2,\dots,t-1$. Then there is a unique numerical semigroup $S$ such that $$R=S\boxplus_{C}\{0,t+1,\rightarrow\}\boxplus_{B}S^{*}=S\boxplus_{C}\{0,t,\rightarrow\}\boxplus_{E}S^{*} \text{ or } R=S\boxplus_{C}\{0,t-1,\rightarrow\}\boxplus_{C}S^{*}.$$
\end{corollary}

For a given numerical semigroup $R$, we will give conditions on $R$ which guarantees that $R\boxplus_{C}\{0,t+1,\rightarrow\}\boxplus_{B}R^{*}$, $R\boxplus_{B}\{0,t+1,\rightarrow\}\boxplus_{C}R^{*}$ and $R\boxplus_{C}\{0,t-1,\rightarrow\}\boxplus_{C}R^{*}$ are numerical semigroups.

%\begin{theorem} \cite[Theorem 4.14]{SY}
%Let $R=\{ 0,r_1,\dots,r_n,\rightarrow\}$ be a numerical semigroup with $\G(R)=\{ a_1,\dots,a_g \}$. Then
%\begin{enumerate}
%\item $R\boxplus_{C}\{0,3,\rightarrow\}\boxplus_{B}R^{*}=R\boxplus_{C}\{0,2,\rightarrow\}\boxplus_{E}R^{*}$ is a numerical semigroup if and only if $r_n$ is a minimal generator for $R$ and $2r_n-r_i-r_j\neq r_k$ for all $i,j,k\in\{0,\dots,n-1\}$.
%\item $R\boxplus_{C}\{0,2,\rightarrow\}\boxplus_{B}R^{*}=R\boxplus_{E}R^{*}$ is a numerical semigroup if and only if $2r_n-r_i-r_j-1\neq r_k$ for all $i,j,k\in\{0,\dots,n-1\}$.
%\item $R\boxplus_{C}R^{*}$ is a numerical semigroup if and only if $r_n$ is a minimal generator for $R$ and $2a_g-r_i-r_j\neq r_k$ for all $i,j,k\in\{0,\dots,n-1\}$.
%\end{enumerate}
%\end{theorem}

\begin{theorem}
	Let $R=\{ 0,r_1,\dots,r_n,\rightarrow\}$ be a numerical semigroup with $\G(R)=\{ a_1,\dots,a_g \}$ and $t>2$. Then
	\begin{enumerate}
		\item $R\boxplus_{C}\{0,t+1,\rightarrow\}\boxplus_{B}R^{*}=R\boxplus_{C}\{0,t,\rightarrow\}\boxplus_{E}R^{*}$ is a numerical semigroup if and only if $r_n,r_n+1,\dots,r_n+t-2$ are minimal generators for $R$ where $t<r_{n-1}+r_1-r_n+2$, and $2r_n-r_{i}-r_{j}+t-2\neq r_k$ $\forall i,j,k\in[0,n-1]\cap\N_0$.
		\item $R\boxplus_{B}\{0,t+1,\rightarrow\}\boxplus_{C}R^{*}$ is a numerical semigroup if and only if $r_n,r_n+1,\dots,r_n+t-1$ are minimal generators for $R$ where $t<r_{n-1}+r_1-r_n+1$, as well as  $r_n-r_{i}+t-1\neq r_k$ and $2r_n-r_{i}-r_{j}+t-2\neq r_k$ $\forall i,j,k\in[0,n-1]\cap\N_0$.
		\item $R\boxplus_{C}\{0,t-1,\rightarrow\}\boxplus_{C}R^{*}$ is a numerical semigroup if and only if $r_n,r_n+1,\dots,r_n+t-3$ are minimal generators for $R$ where $t<r_{n-1}+r_1-r_n+3$, and $2r_n-r_{i}-r_{j}+t-4\neq r_k$ $\forall i,j,k\in[0,n-1]\cap\N_0$.
	\end{enumerate}
\end{theorem}

\begin{proof}
	
	We prove the theorem using hook length properties of $Y_R$ and $Y_{R^{*}}$. Note that we will make use of $\G(\{0,t,\rightarrow\})=\{1,\dots,t-1\}$, $R^{*}=\{ 0,a_g-a_{g-1},\dots,a_g-a_1,r_n,\rightarrow \}$ and $\G(R^{*})=\{ a_g-r_{n-1},\dots,a_g-r_1,a_g \}$.
	
	By definition, $R\boxplus_{C}\{0,t,\rightarrow\}=\{ 0,r_1,\dots,r_{n-1},r_n+t-1,\rightarrow\}$ and $\G(R\boxplus_{C}\{0,t,\rightarrow\})=\{a_1,\dots,a_g,r_n,\dots,r_n+t-2\}$. Remember that $Y_R$ has $g$ rows and $n$ columns. Thus, the associated Young diagram $Y_{R\boxplus_{C}\{0,t,\rightarrow\}}$ of $R\boxplus_{C}\{0,t,\rightarrow\}$ is obtained by adding $t-1$ rows of length $n$ on top of $Y_R$. This corresponds to the $(t-1)\times n$ matrix shaped and orange part of Figure \ref{fig}, and the hook lengths of boxes in this part are labeled with $u_{ij}$ where $i\in\{ 1,\dots,t-1 \}$ and $j\in\{ 1,\dots,n \}$. Then by Proposition \ref{prp}, we see that $u_{ij}=r_n-r_{j-1}+t-i-1$ for every $i$ and $j$.
	
	Now, by definition
	\begin{align*}
		R\boxplus_{C}\{0,t,\rightarrow\}\boxplus_{E}R^{*}=&\{ 0,r_1,\dots,r_{n-1},r_n+t-1,r_n+t-1+a_g-a_{g-1},\dots\\
		&\hspace{2cm}\dots,r_n+t-1+a_g-a_1,r_n+t-1+r_n,\rightarrow\}\\
		=&\{ 0,r_1,\dots,r_{n-1},r_n+t-1,2r_n+t-2-a_{g-1},\dots\\
		&\hspace{2cm}\dots,2r_n+t-2-a_1,2r_n+t-1,\rightarrow\}
	\end{align*}
	and
	\begin{align*}
		\G(R\boxplus_{C}\{0,t,\rightarrow\}\boxplus_{E}R^{*})=&\{ a_1,\dots,a_g,r_n,\dots,r_n+t-2,r_n+t-2+a_g-r_{n-1}+1,\dots\\&\hspace{2cm}\dots,r_n+t-2+a_g-r_1+1,r_n+t-2+a_g+1 \}\\
		=&\{ a_1,\dots,a_g,r_n,\dots,r_n+t-2,2r_n-r_{n-1}+t-2+,\dots\\&\hspace{2cm}\dots,2r_n-r_1+t-2,2r_n+t-2 \}
	\end{align*}
	
	Therefore, the associated Young diagram $Y_{R\boxplus_{C}\{0,t,\rightarrow\}\boxplus_{E}R^{*}}$ of $R\boxplus_{C}\{0,t,\rightarrow\}\boxplus_{E}R^{*}$ is obtained by adding $n$ rows of length $n$ on top of $Y_{R\boxplus_{C}\{0,t,\rightarrow\}}$ and then glueing $Y_{R^{*}}$ to the right of it as shown in Figure \ref{fig}. The hook lengths of boxes in the top left $n\times n$ matrix shaped and white part of Figure \ref{fig} are labeled with $\mu_{ij}$. Then by Proposition \ref{prp} again, we see that $\mu_{ij}=2r_n-r_{i-1}-r_{j-1}+t-2$ for every $i,j\in [1,n]\cap\N$.
	
	\begin{figure}
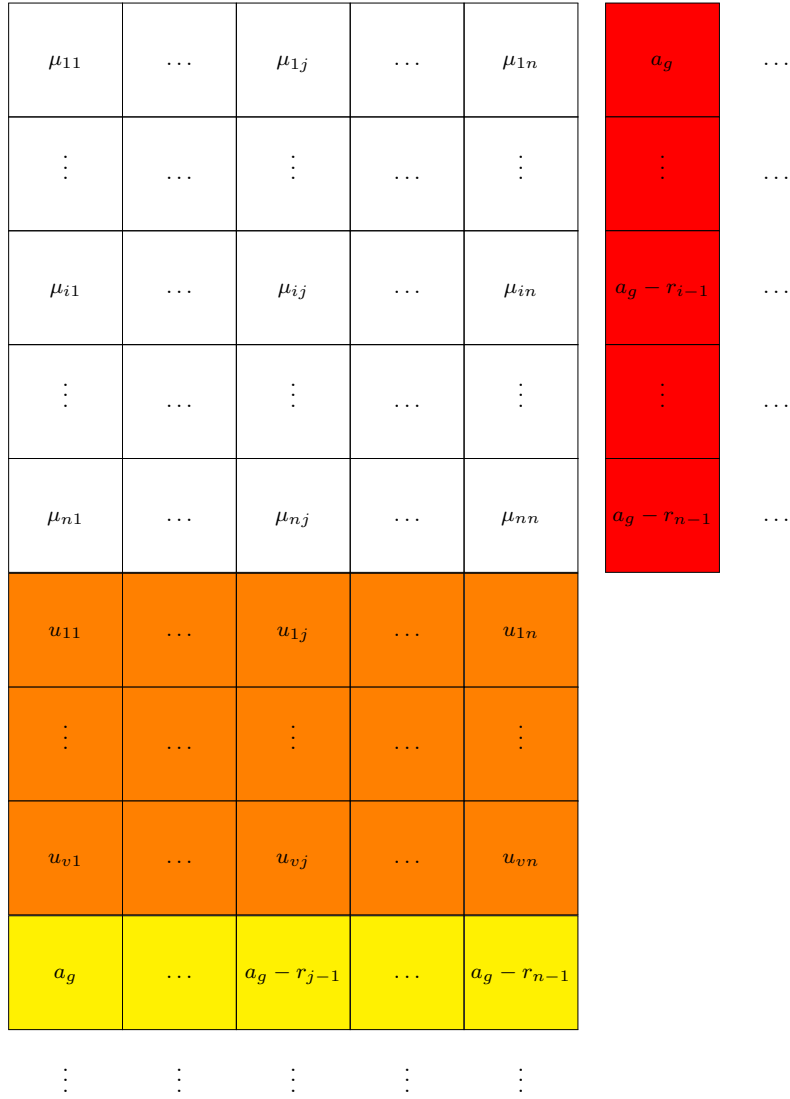

		\begin{center}
			{\tiny\ytableausetup{mathmode, boxsize=5em}
				\begin{tabular}{ccccc}
					\begin{ytableau}
						\mu_{11} & \dots & \mu_{1j} & \dots & \mu_{1n}\\
						\vdots & \dots & \vdots & \dots & \vdots\\
						\mu_{i1} & \dots & \mu_{ij} & \dots & \mu_{in}\\
						\vdots & \dots & \vdots & \dots & \vdots\\
						\mu_{n1} & \dots & \mu_{nj} & \dots & \mu_{nn}\\
					\end{ytableau}&
					\begin{ytableau}[*(red)]
						a_{g}&\none[\dots]\\ \vdots&\none[\dots]\\ a_{g}- r_{i-1}&\none[\dots]\\ \vdots&\none[\dots]\\ a_{g}-r_{n-1}&\none[\dots]
					\end{ytableau}\\
					
					\begin{ytableau}[*(orange)]
						u_{11}  & \dots & u_{1j} & \dots & u_{1n}\\
						\vdots & \dots & \vdots & \dots & \vdots\\
						u_{v1}  & \dots & u_{vj} & \dots & u_{vn}\\
						
					\end{ytableau}\\
					
					\begin{ytableau}[*(yellow)]
						a_{g} & \dots & a_{g}- r_{j-1} & \dots & a_{g}-r_{n-1}\\
						\none[\vdots]& \none[\vdots]& \none[\vdots]&\none[\vdots]&\none[\vdots]
					\end{ytableau}
			\end{tabular}}	
		\end{center}
		\caption{Young diagram of $R\boxplus_{C}\{0,t+1,\rightarrow\}\boxplus_{B}R^{*}=R\boxplus_{C}\{0,t,\rightarrow\}\boxplus_{E}R^{*}$}
		\label{fig}
	\end{figure}
	
	By Proposition \ref{prp}, we also know that $R\boxplus_{C}\{0,t,\rightarrow\}\boxplus_{E}R^{*}$ is a numerical semigroup if and only if hook lengths of all boxes in $Y_{R\boxplus_{C}\{0,t,\rightarrow\}\boxplus_{E}R^{*}}$ are in $\G(R\boxplus_{C}\{0,t,\rightarrow\}\boxplus_{E}R^{*})$. Since $R$ is a numerical semigroup, hook lengths of boxes in the yellow part (associated to $Y_R$) and the red part (associated to $Y_{R^{*}}$) of Figure \ref{fig} are in $\G(R\boxplus_{C}\{0,t,\rightarrow\}\boxplus_{E}R^{*})$. Thus, we only need to check if $u_{ij}$'s and $\mu_{ij}$'s are in $\G(R\boxplus_{C}\{0,t,\rightarrow\}\boxplus_{E}R^{*})$.
	
	However, we have
	\begin{align*}
		u_{ij}\in\G(R\boxplus_{C}\{0,t,\rightarrow\}\boxplus_{E}R^{*})\iff& u_{ij}\in\G(R\boxplus_{C}\{0,t,\rightarrow\})\\ \iff& R\boxplus_{C}\{0,t,\rightarrow\} \text{ is a numerical semigroup}
	\end{align*}
	
	This means that $u_{i1}=r_n+t-i-1$ can not be written as a positive factorization of $\{ 0,r_1,\dots,r_{n-1}\}$, and so $u_{i1}$ is a minimal generator of $R$ for each $i\in[1,t-1]\cap\N$. On the other hand, $u_{11}=r_n+t-2 <r_{n-1}+r_1$ implies that $t<r_{n-1}+r_1-r_n+2$. Otherwise, one of $u_{i1}$ must be an element of $R\boxplus_{C}\{0,t,\rightarrow\}$, which is clearly a contradiction.
	
	Further $\mu_{ij}\in \G(R\boxplus_{C}\{0,t,\rightarrow\}\boxplus_{E}R^{*})$ if and only if $2r_n-r_{i}-r_{j}+t-2\neq r_k$, $2r_n-r_{i}-r_{j}+t-2\neq r_n+t-1$ and $2r_n-r_{i}-r_{j}+t-2\neq2r_n-a_l+t-2$ for every $i,j,k\in[0,n-1]\cap\N_0$ and $l\in[1,g-1]\cap\N$. However, since $R$ is a numerical semigroup, $$2r_n-r_{i}-r_{j}+t-2\neq r_n+t-1 \iff r_n-1\neq r_i+r_j$$ and $$2r_n-r_{i}-r_{j}+t-2\neq2r_n-a_l+t-2 \iff r_i+r_j\neq a_l$$ always hold for every $i,j,k\in[0,n-1]\cap\N_0$ and $l\in[1,g-1]\cap\N$. Hence, $$\mu_{ij}\in \G(R\boxplus_{C}\{0,t,\rightarrow\}\boxplus_{E}R^{*}) \iff 2r_n-r_{i}-r_{j}+t-2\neq r_k$$ for every $i,j,k\in[0,n-1]\cap\N_0$. This proves \textit{(1)}, and others follow in a similar way.
	
\end{proof}

\subsection*{Acknowledgements} 
The author completed most of this work while visiting Manuel Delgado at the Center of Mathematics of the University of Porto (CMUP). The author gratefully acknowledges Manuel Delgado for his generous support, particularly regarding GAP programming tools, and extends sincere thanks to CMUP for its kind hospitality.

\end{document}